\documentclass[11pt]{article}
\usepackage{amsthm,amsfonts,amsmath,amscd,amssymb,epsfig}

\usepackage{enumerate,array}
\theoremstyle{plain}
\newtheorem{theorem}{Theorem}
\newtheorem{thm}[theorem]{Theorem}

\newtheorem{lemma}[theorem]{Lemma}

\theoremstyle{remark}

\newtheorem{remark}[theorem]{Remark}
\newtheorem{example}[theorem]{Example}

\theoremstyle{definition}

\newcommand{\p}{\partial}
\newcommand{\pb}{\bar\partial}

\newcommand{\N}{{\mathbb{N}}}
\newcommand{\Z}{{\mathbb{Z}}}
\newcommand{\R}{{\mathbb{R}}}
\newcommand{\T}{{\mathbb{T}}}
\newcommand{\C}{{\mathbb{C}}}
\newcommand{\Q}{{\mathbb{Q}}}

\newcommand{\Aut}{{\rm Aut}}

\newcommand{\ind}{{\rm ind}}
\newcommand{\im}{{\rm im }}        

\newcommand{\CZ}{{\rm CZ}}

\newcommand{\JJ}{\mathcal{J}}
\newcommand{\MM}{\mathcal{M}}

\renewcommand{\AA}{\mathcal{A}}

\newcommand{\comment}[1]{}

\newcommand{\lin}{{\mathrm{lin}}}

\newcommand{\pa}{\partial}
\newcommand{\dist}{\operatorname{dist}}

\newcommand{\krn}{\operatorname{ker}}

\newcommand{\Spa}{\operatorname{Span}}
\title{A note on Reeb dynamics on the tight 3-sphere}
\author{F.~Bourgeois, K.~Cieliebak, and T.~Ekholm}
\date{5 July 2007}

\begin{document}
\maketitle
\abstract{We show that a nondegenerate tight contact form on the $3$-sphere has
  exactly two simple closed Reeb orbits if and only if the
  differential in linearized contact homology
  vanishes. Moreover, in this case
  the Floquet multipliers and Conley-Zehnder indices of the two
  Reeb orbits agree with those of a suitable irrational ellipsoid in
  $4$-space.}
\parindent=0pt
\parskip=4pt

\section{Introduction}\label{sec:intro}

A {\em contact form} on a closed $3$-manifold $Y$ is a $1$-form
$\lambda$ such that $\lambda\wedge d\lambda$ is a volume form on
$Y$. The {\em contact structure} determined by a contact form
$\lambda$ is the tangent hyperplane field   $\krn(\lambda)\subset
TY$. The condition on $\lambda\wedge d\lambda$ guarantees that the
contact structure is a completely non-integrable plane field. The {\em
  Reeb vector field} determined by the contact form $\lambda$ is the
vector field $R_\lambda$ on $Y$ uniquely determined by the conditions
$\lambda(R_\lambda)=1$ and $\iota_{R_\lambda}d\lambda=0$. A periodic
solution $\gamma$ of the differential equation determined by $R_\lambda$ is
called a {\em (closed) Reeb orbit}.

An {\em overtwisted disk} in a contact $3$-manifold $Y$ is an embedded
$2$-disk $D\subset Y$ such that the foliation of $D$ induced by the
contact structure has exactly one singular point and such that the
boundary $\pa D$ of $D$ is a closed leaf in this foliation. A contact
$3$-manifold which does not contain any overtwisted disk is called {\em
  tight}.

For a star-shaped (with respect to the origin) hypersurface
$Y\subset\R^4\cong\C^2$ the canonical 1-form
$\lambda=\frac{1}{2}\sum_{j=1}^2(x_jdy_j-y_jdx_j)$, where
$(z_1,z_2)=(x_1+iy_1,x_2+iy_2)$ are linear coordinates on $\C^{2}$,
restricts to a
contact form on $Y$. This contact form is tight, and by Eliashberg's
uniqueness theorem~\cite{El} each tight contact form on the
three-sphere $S^3$ arises in this way. Particularly simple
hypersurfaces of this form are the {\em irrational ellipsoids}
$$
E(a_1,a_2)=\left\{(z_1,z_2)\in\C^2\left|\,
\frac{|z_1|^2}{a_1}+\frac{|z_2|^2}{a_2}=1\right.\right\}
$$
for $a_1,a_2>0$ and $a_1/a_2\notin \Q$.

In~\cite{EGH} Eliashberg, Givental, and Hofer introduced Symplectic Field
Theory (SFT). It is a framework for extracting invariants of contact
and symplectic manifolds via holomorphic curve counts.

For the tight 3-sphere there is a particularly simple SFT type
invariant called ``linearized contact homology''. For the sake of
completeness, we will recall its definition and a proof of its invariance
in Section~\ref{sec:SFT}. Roughly speaking, linearized contact homology is defined as
follows. Let $\lambda$ be a contact form on $S^3$ such that all closed
Reeb orbits are {\em nondegenerate}, i.e.~no {\em Floquet multiplier}
(eigenvalue of the linearized return map on a transverse section)
equals $1$. Then each
closed Reeb orbit has a well-defined {\em Conley-Zehnder index}
$\CZ(\gamma)\in\Z$, see~\cite{HWZ}. Following~\cite{HWZ}, we call a
contact form $\lambda$ {\em dynamically convex} if all closed Reeb
orbits are nondegenerate and have Conley-Zehnder index at least
$3$. E.g.~this is the case for the induced contact forms on irrational ellipsoids,
and more generally for the contact forms on hypersurfaces in $\R^4$ bounding strictly convex
domains~\cite{HWZ}. For a dynamically convex contact form $\lambda$
let $CC_*(S^3,\lambda)$ be the $\Q$-vector space generated by the ``good''
(see Section~\ref{sec:SFT}) closed Reeb orbits, graded by their {\em degree}
$|\gamma|=\CZ(\gamma)-1$. Fixing an $\R$-invariant almost
complex structure $J$ on $\R\times S^3$ compatible with $\lambda$, we
define a differential $\p:CC_*(S^3,\lambda)\to CC_{*-1}(S^3,\lambda)$
by counting rigid $J$-holomorphic cylinders connecting Reeb
orbits. We show in Section~\ref{sec:SFT} that $\p^2=0$, that the {\em
linearized contact homology} $HC^\lin(S^3)=\ker\p/\im\p$ is
independent of the choice of $J$ and of $\lambda$, and furthermore that
$$
   HC_k^\lin(S^3) = \begin{cases}
   \Q & \text{for $k\geq 2$ even,} \cr
   0 & \text{otherwise}.
   \end{cases}
$$
In this note we prove the following result.

\begin{thm}\label{thm:two}
Let $\lambda$ be a dynamically convex tight contact form on the
3-sphere. Then the following are equivalent:
\begin{itemize}
\item[{\rm (i)}] The differential in linearized contact homology vanishes.
\item[{\rm (ii)}] There are precisely two simple Reeb orbits
 $\gamma_1,\gamma_2$.
\end{itemize}
Moreover, in this case $\gamma_1$ and $\gamma_2$ are unknotted, elliptic,
have linking number $1$, and their Conley-Zehnder indices and Floquet
multipliers agree with those of a suitable irrational ellipsoid.
\end{thm}


\begin{remark}\label{r:sp=gen}
The abstract perturbation theory developed in~\cite{HWZ-poly} will
eventually lead to a definition and an invariance proof of linearized contact homology for arbitrary
(not necessarily dynamically convex) tight contact forms on $S^3$ with nondegenerate closed Reeb orbits,
cf.~Remark~\ref{rem:SFT} below. Using this in combination with formulas for Conley-Zehnder indices, see Section \ref{ssec:CZ}, it is immediate that condition (i) implies dynamical convexity and an additional argument, see Remark \ref{r:dynconv}, shows that (ii) does as well. Thus, the restriction in Theorem \ref{thm:two} to the technically simpler dynamically convex case turns out not to be any restriction at all.
\end{remark}

\begin{remark}
The implication that if there are precisely two simple closed Reeb
orbits $\gamma_1,\gamma_2$ then their Floquet multipliers lie on the
unit circle with irrational angles has been proved independently
in~\cite{HLW}.
\end{remark}

\begin{remark}
It follows from~\cite{HWZ} that in the situation of
Theorem~\ref{thm:two} the closed Reeb orbit of degree 2
is the binding of an open book decomposition with pages whose
interiors are transverse to the Reeb vector field. The return map of a
page is an area preserving diffeomorphism of the open disk with one
fixed point and no other periodic points. However, the return map need
not be conjugate to an irrational rotation (see~\cite{FK}).
\end{remark}

\begin{remark}
In view of the result in~\cite{HWZ}, Theorem~\ref{thm:two}
implies the following dichotomy for the Reeb dynamics on a
strictly convex hypersurface in $\R^4$. Either the differential in
linearized contact homology vanishes and there are precisely two
simple closed orbits; or the differential does not vanish and
there are infinitely many simple closed orbits. Moreover,
generically the second case occurs (even in the class of star-shaped
hypersurfaces).

This discussion motivates the following conjecture. {\em If, for a
  star-shaped hypersurface in $\R^4$, the differential in linearized
  contact homology does not vanish, then it carries infinitely many
  closed characteristics.} More optimistically, one could even
  conjecture that a star-shaped hypersurface in $\R^{2n}$ carries
  either precisely $n$ (if the differential in linearized contact
  homology vanishes) or infinitely many (if it does not vanish) simple
  closed Reeb orbits. See~\cite{Lo} for an exposition of known
  multiplicity results.
\end{remark}

{\bf Acknowledgements. }
FB was partially supported by the Fonds National de la Recherche
Scientifique, Belgium.

KC was partially supported by DFG grant CI 45/2-2.

TE acknowledges support from the Royal Swedish Academy of Sciences, Research
Fellow sponsored by the Knut and Alice Wallenberg foundation, from the
Alfred P. Sloan Foundation, Research  Fellow, and from NSF-grant
DMS-0505076.

We thank E.~Volkov for helpful comments, and
D.~Kotschick for pointing out Remark~\ref{rem:jump}.

\section{Linearized contact homology}\label{sec:SFT}

In this section we define the linearized contact homology for a
dynamically convex contact form on $S^3$ and prove its
invariance. See~\cite{EGH} for details of the setup.

Fix a dynamically convex contact form $\lambda$ on $S^3$. If $\gamma$
is a simple closed orbit of the Reeb field $R_\lambda$,
we denote by $\gamma^k$ its $k$-th iterate. We call $\gamma^k$
{\em good} if the Conley-Zehnder indices (see Section~\ref{sec:proof})
of $\gamma$ and $\gamma^k$ have the same parity. Otherwise, we call
$\gamma^k$ {\em bad}.
Let $CC_*(S^3,\lambda)$ be the $\Q$-vector space generated by the good
closed Reeb orbits, graded by their degree $|\gamma|=\CZ(\gamma)-1$.

An $\R$-invariant almost complex structure $J$ on $\R\times S^3$ is
called {\em compatible with $\lambda$} if it preserves
$\xi=\ker\lambda$, if $d\lambda(\cdot,J\cdot)$ defines a metric on $\xi$,
and if $J \frac{\partial}{\partial t} = R_\lambda$, where $t$ denotes the
coordinate on $\R$. Fix such a $J$.
For good closed orbits $\gamma, \gamma_1, \ldots, \gamma_r$ of
$R_\lambda$ of periods $T, T_1, \ldots, T_r$,
let $\mathcal{M}(\gamma;\gamma_1, \ldots, \gamma_r)$ be the moduli
space consisting of equivalence classes of tuples
$(x,p;y_1,p_1\dots,y_r,p_r;f)$, where $x,y_1,\dots,y_r$ are distinct
points on $S^2$ with directions $p,p_1,\dots,p_r$ and $f =(a,u) : S^2
\setminus \{ x, y_1, \ldots, y_r \} \to \R \times S^3$ is a map with
the following properties: 
\begin{itemize}
\item $df+ J\circ df \circ j = 0$,
\item $\lim_{z \to x} a(z) = +\infty$ and $\lim_{z \to y_i} a(z) = -\infty$,
\item in polar coordinates $(\rho, \theta) \in ]0,1] \times \R/\Z$
around $x$ in which $p$ corresponds to $\theta=0$ we have $\lim_{\rho\to
  0} u(\rho, \theta) = \gamma(-T \theta)$, and 
\item in polar coordinates $(\rho_i, \theta_i)
\in ]0,1] \times \R/\Z$ around $y_i$ in which $p_i$ corresponds to
$\theta_i=0$ we have $\lim_{\rho_i\to 0} u(\rho_i, \theta_i) =
\gamma_i(T_i \theta_i)$.
\end{itemize}
Two such tuples $(x,p;y_1,p_1\dots,y_r,p_r;f)$ and
$(x',p';y_1',p_1'\dots,y_r',p_r';f')$ are equivalent
if there exists a biholomorphism $h$ of $S^2$ such that $h(x) = x'$,
$h(y_i) = y'_i$, $d_xh\cdot p=p'$, $d_{y_i}h\cdot p_i=p_i'$ and $f =
f' \circ h$. Since $J$ is $\R$-invariant, 
$\R$ acts on these moduli spaces by translation and we denote the quotient by
$\mathcal{M}(\gamma;\gamma_1, \ldots, \gamma_r)/\R$. Using an
appropriate functional analytic setup, the moduli spaces can be
described as the zero locus of a Fredholm section of a certain bundle
and have expected dimension (determined by the Fredholm index) given by
$$
   \dim\bigl(\mathcal{M}(\gamma;\gamma_1, \ldots, \gamma_r)/\R\bigr) =
   |\gamma|-\sum_{j=1}^r|\gamma_j|\,\,-\,1.
$$
If this dimension is zero and if the moduli space is compact and cut out
transversally, then it consists of finitely many points. One can
associate a sign to each of these points via coherent orientations on
the moduli spaces~\cite{BM} and we denote by $n(\gamma;\gamma_1, \ldots,
\gamma_r)$ the algebraic count of the elements in
$\mathcal{M}(\gamma;\gamma_1, \ldots, \gamma_r)/\R$. Furthermore, we
denote by $\kappa_\gamma$ the multiplicity of the Reeb orbit
$\gamma$. Define the linear map 
$$
   \p:CC_*(S^3,\lambda)\to CC_{*-1}(S^3,\lambda),\qquad
   \gamma\mapsto \sum_{|\gamma'|=|\gamma|-1}
   \frac{n(\gamma;\gamma')}{\kappa_{\gamma'}}\gamma'. 
$$
Thus $\p$ counts rigid $J$-holomorphic cylinders interpolating between
closed Reeb orbits $\gamma$ and $\gamma'$.

\begin{thm}\label{thm:SFT}
Let $\lambda$ be a dynamically convex contact form on $S^3$.
Then for a generic $S^1$-dependent compatible almost complex structure
$J$ the map $\p$ is well-defined and satisfies $\p^2=0$. Moreover, the
{\em linearized contact homology} $HC^\lin(S^3)=\ker\p/\im\p$ is
independent of the choice of $J$ and $\lambda$, and is given by
$$
   HC_k^\lin(S^3) = \begin{cases}
   \Q & \text{for $k\geq 2$ even,} \cr
   0 & \text{otherwise}.
   \end{cases}
$$
\end{thm}

\begin{proof}
The proof follows the familiar scheme from Floer homology (see
e.g.~\cite{Sa, EGH}), provided we can prove that 1-dimensional moduli
spaces of holomorphic cylinders are regular and compact up to breaking
into pairs of cylinders. As usual, one needs to prove this in three
cases:
\begin{itemize}
\item in a symplectization (to establish \cite[Proposition 1.9.1]{EGH}),
\item in a cobordism (to establish \cite[Proposition 1.9.3]{EGH}), and
\item for a homotopy of almost complex structures on a cobordism (to
  establish \cite[Proposition 1.9.4]{EGH}).
\end{itemize}
We will explain the argument in the case of a homotopy, the other two
cases being analogous but easier. Once invariance of $HC^\lin_k(S^3)$
is established it can be easily computed for an irrational ellipsoid
using the formulae for the Conley-Zehnder indices in
Section~\ref{sec:proof} ($\p=0$ in this case).

\smallskip
{\bf Step 1:} Let $Y=S^3$. Consider $X=\R\times Y$ with an exact symplectic form
$d\lambda$ which coincides near $\{\pm\infty\}\times Y$ with
$d(e^t\lambda_\pm)$ for dynamically convex contact forms $\lambda_\pm$
on $Y$. Let $(J_\tau)_{\tau\in[0,1]}$ be a homotopy of almost complex
structures on $\R\times Y$ such that for every $\tau$, $J_{\tau}$ is compatible with
$d\lambda$ and coincides near $\pm\infty$ with fixed $\R$-invariant
almost complex structures $J_\pm$ compatible with $\lambda_\pm$. Fix
closed Reeb orbits $\gamma_\pm$ for $\lambda_\pm$ of equal degrees
$|\gamma_+|=|\gamma_-|$ and consider the 1-dimensional moduli space
$$
   \MM_{[0,1]}=\cup_{\tau\in[0,1]}\MM(\gamma_+;\gamma_-;J_\tau).
$$
By the SFT compactness theorem~\cite{BEHWZ,CM-comp}, every sequence
$f_k$, $k=1,2,\dots$ in $\MM_{[0,1]}$ has a subsequence
converging as $k\to\infty$ to a {\em broken holomorphic sphere}
$F=\{F_\alpha\}_{\alpha\in T}$. Here $T$ is a directed tree (i.e.~each
edge is directed) with the
following properties. For each vertex $\alpha\in T$, $F_\alpha$ is a
punctured holomorphic sphere with exactly one positive puncture and
any number of negative punctures in
$(\R\times Y,J_+)$, in $(\R\times Y,J_-)$, or in $(X,J_\tau)$ for some
$\tau\in[0,1]$. Each edge $e$ of $T$ is labeled by a closed Reeb orbit
$\gamma_e$ of $\lambda_\pm$. If $e$ is directed from a vertex
$\alpha$ to a vertex $\beta$, then $\gamma_e$ is the asymptotic Reeb
orbit at the unique positive puncture of $F_\alpha$ and at one negative
puncture of $F_\beta$. Conversely, each puncture on any $F_\alpha$
corresponds to a unique edge in this way except for two {\em free
punctures}, a positive one asymptotic to $\gamma_+$ and a negative one
asymptotic to $\gamma_-$.

Since the Fredholm index is additive under joining spheres at Reeb
orbits and since the only two free (not paired across edges) punctures
of $F$ are asymptotic to $\gamma_\pm$, the expected dimensions
(i.e. the Fredholm indices) $\ind(F_\alpha)$ of the moduli spaces of
$J_\pm$-, or
$J_\tau$-holomorphic spheres ($\tau$ fixed) which contain the $F_\alpha$ satisfy
$$
   \sum_{\alpha\in T}\ind(F_\alpha) = |\gamma_+|-|\gamma_-| = 0.
$$

The structure of the tree $T$ can be described as follows. Let
$\alpha_\pm\in T$ be the
vertices such that $F_{\alpha_\pm}$ contains the free punctures asymptotic to
$\gamma_\pm$. Define the {\em stem} $S$ of the tree $T$ to be the unique linear
(i.e.~at most two edges meet at each vertex) subtree $S\subset T$
connecting $\alpha_+$ and $\alpha_-$. Define the {\em branches}
$B_1,\dots,B_k$ to be the connected components of $T-S$. Since each
branch $B_i$ has precisely one free  puncture which is
positive and asymptotic to a Reeb orbit $\gamma_i$, its total index
$\ind(B_i)=\sum_{\alpha\in B_i}\ind(F_\alpha)$ satisfies
$$
   \ind(B_i) = |\gamma_i| \geq 2
$$
by dynamical convexity.

The {\em orbit cylinder} over a closed Reeb orbit $\gamma$ for
$\lambda_\pm$ is the $J_\pm$-holomorphic cylinder
$\R\times\gamma\subset\R\times Y$. We call a punctured holomorphic
sphere {\em good} if it is not a branched cover of an orbit cylinder.
Now let us assume that the following regularity condition holds.

\begin{itemize}
\item[(R)] All good components $F_\alpha$ with $\alpha\in S$ in the stem are
{\em regular},  i.e.~transversely cut out by the $1$-parameter family
of Cauchy-Riemann
operators corresponding to the homotopy $(J_\tau)_{\tau\in[0,1]}$.
\end{itemize}

We conclude that each good $F_\alpha$ in $(\R\times Y,J_\pm)$ is
transversely cut out and thus satisfies $\ind(F_\alpha)\ge 1$ because
of translation invariance (see Step 2 below for a full explanation of
how to reach this conclusion). Moreover, there is a finite collection of
exceptional $\tau$-values $0<\tau_1<\dots<\tau_m<1$ such that the
following hold. For every $\tau\ne \tau_j$, $j=1,\dots,m$, each good
$F_\alpha$ in $(X,J_\tau)$ and in the stem belongs to a moduli space
which is transversely cut out by the $\bar\pa_{J_\tau}$-equation and
hence $\ind(F_\alpha)\ge 0$. For $\tau=\tau_j$ some
$j\in\{1,\dots,m\}$, there exists a unique exceptional good punctured
holomorphic sphere
$C_\tau\in\MM(\beta_0;\beta_1,\dots,\beta_r;J_\tau)$, which may belong
to the stem, with the property that the linearized
$\bar\pa_{J_\tau}$-operator at $C_\tau$ has $1$-dimensional cokernel;
at all other good spheres $F_\alpha\ne C_\tau$ in the stem, the linearized
$\bar\pa_{J_\tau}$-operator is surjective. Since the cokernel has
dimension $1$, it follows that $\ind(C_\tau)\ge -1$.

By Lemma~\ref{lem:ind} below, a branched cover $F_\alpha$ over an
orbit cylinder has index $\ind(F_\alpha)\geq 0$. So we see that
$\ind(F_\alpha)\geq 0$ for each component $F_\alpha$ in the stem with
$F_\alpha\neq C_\tau$.

Recall that each edge $e\in T$ was labeled with a Reeb orbit
$\gamma_e$. We order the edges by the actions $\AA(\gamma_e)$. For area
reasons, this order strictly increases in the direction of the tree
and therefore the special component $C_\tau$ can occur at most once
among the $F_\alpha$ with $\alpha\in S$. In particular, in view of the
preceding discussion the total index of the stem $S$ satisfies
$$
   \ind(S)\geq -1.
$$
Thus, since $\ind(T)=0$ and since $\ind(B_i)\ge 2$ for any branch
$B_i$, we conclude that there are no branches $B_i$, and hence $T=S$
is a linear tree. In particular, this excludes branched covers of
orbit cylinders and hence all components $F_\alpha$ are good.
Moreover,  there are only the following two possibilities for the
dimensions of the $F_\alpha$, $\alpha\in T$.

{\bf Case 1:} $\ind(F_\alpha)\geq 0$ for all $\alpha\in T$, in which case
$T$ has only one vertex $\alpha$, no breaking occurs, and the
sequence $(f_k)$ converges in $\MM_{[0,1]}$ to $F_\alpha\in\MM_{[0,1]}$.

{\bf Case 2:} $\ind(F_{\beta^\pm})=\pm 1$ for unique vertices
$\beta^\pm\in T$ and $\ind(F_\alpha)=0$ for all vertices $\alpha\in
T-\{\beta^+,\beta^-\}$. Since the linear tree $T$ contains only one
component in $(X,J_\tau)$ and components in $(\R\times Y,J_\pm)$ have
index at least $1$, we conclude that $\beta^\pm$ are the only vertices
in $T$. Hence the limit curve $F$ is a pair of holomorphic cylinders,
which is precisely what is needed for the chain homotopy
property. This concludes the proof modulo the assumption that
regularity condition (R) above holds.
\smallskip

{\bf Step 2:} We adapt the technique used in~\cite{CM-trans} to achieve the
regularity condition (R).

To three distinct points $z_0,z_1,z_2$ on the Riemann sphere
$S^2=\C\cup\{\infty\}$ and a tangent direction $p_0\in T_{z_0}S^{2}/\R_+$ at
$z_0$, where $\R_+$ acts on tangent vectors by scalar multiplication,
we associate an angle $w(z_0,p_0,z_1,z_2)\in S^1$ as follows.
Let $\phi\in\Aut(S^2)$ be the unique M\"obius transformation with
$\phi(z_0)=0$, $\phi(z_1)=1$, and $\phi(z_2)=\infty$, and define
$$
   w(z_0,p_0,z_1,z_2) = d_{z_0}\phi\cdot p_0\in T_0\C/\R_+\cong
   S^1.
$$
The map $w$ is clearly invariant under simultaneous action of
M\"obius transformations on $(z_0,p_0,z_1,z_2)$ and thus it induces a
diffeomorphism
$$
   w:\MM_3^\$\to S^1,
$$
where $\MM^{\$}_3$ is the decorated Deligne-Mumford space of 3 ordered
distinct points
$(z_0,z_1,z_2)$ on the Riemann sphere with a specified direction $p_0$
at $z_0$
(see~\cite{BEHWZ}). One easily sees that this map extends to arbitrary
trees of spheres for which $z_2$ lies between $z_0$ and $z_1$
(i.e.~$z_2$ lies on the unique embedded path connecting $z_0$ and
$z_1$).

Let $\JJ_{S^1}$ be the space of $S^1$-dependent almost complex
structures on $\R\times Y$ of the type considered above, i.e., which
are compatible with
$d\lambda$ and which coincides near $\pm\infty$ with fixed $\R$-invariant
almost complex structures $J_\pm$. Each $J\in\JJ_{S^1}$ induces a
Cauchy-Riemann operator acting on tuples $(z_0,p_0,z_1,f)$ consisting
of distinct points $z_0,z_1\in
S^2$, a direction $p_0$ at $z_0$, and a map
$f:S^2-\{z_0,z_1\}\to\R\times Y$ by
$$
   \pb_J(z_0,p_0,z_1,f)(z) =
   \frac{1}{2}\Bigl(df(z)+J\bigl(w(z_0,p_0,z_1,z)\bigr)\circ df(z)\circ
   i\Bigr),
$$
where $z\in S^2-\{z_0,z_1\}$. This operator is clearly invariant under the
simultaneous action of $\Aut(S^2)$ on $(z_0,p_0,z_1,f)$. For fixed
$(z_0,p_0,z_1)$, after applying a M\"obius transformation that sends
$z_0$ to $0$, $z_1$ to $\infty$ and $p_0$ to $\R_+$, we obtain a
Cauchy-Riemann operator on maps $f\colon\R\times(\R/2\pi)\to\R\times
Y$ where we think of the source as an infinite cylinder (using
$\C-\{0\}\cong\R\times(\R/2\pi)$). In polar coordinates
$(s,t)\in\R\times(\R/2\pi\Z)$ this operator is given by
$$
   \pb_J(f)(s,t)
   = \frac{1}{2}\Bigl(df(s,t)+J(e^{-it})\circ df(s,t)\circ i\Bigr),
$$
since $w(0,\R_+,\infty,e^{s+it})=e^{-it}$. A standard argument (see
e.g.~\cite{Sa}) shows that regularity for $1$-parameter families of
such Cauchy-Riemann operators can be achieved by choosing a generic
path of $S^1$-dependent $J$ in $\JJ_{S^1}$. More precisely, for a
Baire set of paths $(J_\tau)_{\tau\in[0,1]}\in \JJ_{S^{1}}$ the
following holds: at any pair $(f,\tau)$ such that
$\pb_{J_\tau}f=0$ and such that $f$ is not a branched cover
of an orbit cylinder, the linearization of
$\bar\pa_{J_\tau}$ at $(f,\tau)$ is surjective.

We use the regularity result for $S^{1}$-dependent almost complex
structures to establish condition (R) as follows. Fix a point
$p_{\bar\gamma}$ on each simple closed Reeb orbit $\bar\gamma$. For
closed Reeb orbits $\gamma_0,\gamma_1,\dots,\gamma_k$ with underlying
simple orbits $\bar\gamma_i$ and $J\in\JJ_{S^1}$ denote by
$\MM(\gamma_0;\gamma_1,\dots,\gamma_k;J)$ the moduli space of
equivalence classes of tuples $(z_0,p_0,z_1,\dots,z_k,f)$ of the
following form:
\begin{itemize}
\item $z_0,\dots,z_k$ are distinct points in $S^2$;
\item $p_0$ is a tangent direction at $z_0$;
\item $f:S^2-\{z_0,\dots,z_k\}\to\R\times Y$ is a map with
  $\pb_J(z_0,p_0,z_1,f)=0$, which has its positive puncture at $z_0$
  where $f$ is asymptotic to $\gamma_0$ and takes the tangent
  direction $p_0$ to the point $p_{\bar\gamma_0}$, and which has
  negative punctures at $z_1,\dots,z_k$ where $f$ is asymptotic to
  $\gamma_1,\dots,\gamma_k$, respectively;
\item $f$ is not a branched cover of an orbit cylinder.
\end{itemize}
Two tuples $(z_0,p_0,z_1,\dots,z_k,f)$ and
$(z_0',p_0',z_1',\dots,z_k',f')$ are equivalent if
they are related under the natural action of $\Aut(S^2)$.

The regularity result for the Cauchy-Riemann operator above implies
that, for generic paths $(J_\tau)_{\tau\in[0,1]}$ in $\JJ_{S^1}$, any
moduli space
$$
\MM_{[0,1]}=\cup_{\tau\in[0,1]}\MM(\gamma_0;\gamma_1,\dots,\gamma_k;J_\tau)
$$
is cut out transversally and is consequently a manifold of dimension
$$
   \dim(\MM_{[0,1]}) =
   |\gamma_0|-\sum_{j=1}^k|\gamma_j|+1.
$$
Here a manifold of negative dimension is understood to be empty.

In order to establish (R) it remains to study the boundary of a
$1$-dimensional moduli space $\MM_{[0,1]}$ as above. Let
$f_k\in\MM_{[0,1]}$ be a sequence as in Step 1 which converges to a
broken holomorphic sphere $F$ modelled on a tree $T$ with stem $S$.
Then each $\alpha\in S$ lies between the special vertices $\alpha_\pm$
corresponding to the free asymptotic orbits $\gamma_\pm$ (the
asymptotic orbits of $f_k$), so any good holomorphic sphere $F_\alpha$,
$\alpha\in S$, in $(X,J_\tau)$ belong to some moduli space
$\MM(\gamma_0;\gamma_1,\dots,\gamma_k;J_\tau)$ of the type above. Thus
the preceding discussion implies the following regularity properties
for $F_\alpha$. If $F_\alpha$ lies in $(\R\times Y,J_\pm)$ it belongs to a
moduli space $\MM(\gamma_0;\gamma_1,\dots,\gamma_k;J_\pm)$ which is
transversely cut out, so $\ind(F_\alpha)\geq 1$ by $\R$-invariance.
If $F_\alpha$ lies in $(X,J_\tau)$ for some $\tau$ it belongs to a
moduli space $\MM_{[0,1]}$ which is transversely cut out, so
$\ind(F_\alpha)\geq -1$. Moreover, $\ind(F_\alpha)=-1$ occurs only for
finitely many components $C_{\tau_j}$ at parameter values
$\tau_1,\dots,\tau_m$ and all other components have
$\ind(F_\alpha)\geq 0$. This proves that the regularity
condition (R) holds and hence demonstrates Theorem~\ref{thm:SFT}.
\end{proof}

\begin{remark}\label{rem:SFT}
Theorem~\ref{thm:SFT} defines linearized contact homology for
dynamically convex contact forms on $S^3$, which turn out to be
sufficient for the purposes of this paper, see Remark
\ref{r:sp=gen}. Linearized contact homology is expected to exist more
generally for any contact manifold $(Y,\lambda)$ with an exact
symplectic filling $(X,d\lambda)$. Here the boundary map between
closed Reeb orbits $\gamma_\pm$ will count punctured holomorphic
spheres in $\R\times Y$ with one positive puncture asymptotic to
$\gamma_+$ and any number of negative punctures asymptotic to Reeb
orbits $\gamma_-,\gamma_1,\dots,\gamma_k$, together with rigid
holomorphic planes in $X$ asymptotic to the $\gamma_i$,
$i=1,\dots,k$. See~\cite{CL,BEE} for details of this construction.
However, it seems that in this more general situation transversality
cannot be achieved by the method in Theorem~\ref{thm:SFT}, but
requires the use of abstract perturbations. Hofer, Wysocki and Zehnder
are developing a theory of such perturbations called ``polyfold
Fredholm theory''~\cite{HWZ-poly}. However, at the time of this
writing their theory is not yet completed.
\end{remark}

\section{Proof of Theorem~\ref{thm:two}}\label{sec:proof}

In this section we prove Theorem \ref{thm:two}. To that end we
study properties of Reeb orbits in dimension $3$ and dynamical
properties of translations on a flat torus.

\subsection{Conley-Zehnder indices}\label{ssec:CZ}
The grading in the contact homology algebra is induced by
Conley-Zehnder indices, see \cite{HWZ}, \cite{Lo}. Here we
recall some properties of the Conley-Zehnder index in dimension $3$ in
the case of non-degenerate Reeb orbits. Let $\gamma$ be a simple
closed Reeb orbit and denote by $\gamma^k$ its
$k$-th iterate. Then the Floquet multipliers (eigenvalues of the
linearized return map on a transverse section) occur in a pair
$\mu,1/\mu\in\R\setminus\{0,\pm 1\}$ if they are real, or
$\mu,\bar\mu\in S^1\setminus\{\pm 1\}$ if not. If $\gamma$ is a Reeb
orbit we write $\CZ(\gamma)$ for its Conley-Zehnder index. According
to Section 8.1 in~\cite{Lo}, we need to distinguish three cases (the
first case is covered by Theorem 7 and the other two by Theorem 6).

{\bf Elliptic case: }$\gamma$ has a non-real Floquet multiplier
$\mu=e^{2\pi i\alpha}$ with $\alpha\in(0,1)$ irrational. Then
$$
   \CZ(\gamma^k) = 2kr + 2[k\alpha] + 1= 2[k(r+\alpha)] + 1,
$$
for some integer $r\in\Z$, where $[x]$ denotes the largest integer
smaller than or equal to $x$.

{\bf Even hyperbolic case: }$\gamma$ has a positive real Floquet
multiplier $\mu\in(0,1)$. Then
$$
   \CZ(\gamma^k) = 2kr,
$$
for some integer $r\in\Z$.

{\bf Odd hyperbolic case: }$\gamma$ has a negative real Floquet
multiplier $\mu\in(-1,0)$. Then
$$
   \CZ(\gamma^k) = (2r+1)k,
$$
for some integer $r\in\Z$. Note that the even multiplies of $\gamma$
are ``bad'' in the sense of~\cite{EGH} and do not contribute to
contact homology.

As mentioned in Section \ref{sec:intro}, the {\em degree $|\gamma|$}
of a closed Reeb orbit $\gamma$ as a
generator of the contact homology algebra as well as a generator of
the linearized contact homology chain complex is given by
$$
   |\gamma| = \CZ(\gamma)-1.
$$
The following lemma was used in the proof of Theorem~\ref{thm:SFT}
above.

\begin{lemma}\label{lem:ind}
Let $\gamma$ be a simple closed Reeb orbit all of whose iterates are
nondegenerate. Then for any positive integers $k_1,\dots,k_s$
$$
   |\gamma^{k_1+\dots+k_s}|\geq|\gamma^{k_1}|+\dots+|\gamma^{k_s}|.
$$
\end{lemma}

\begin{proof}
Set $k=k_1+\dots+k_s$. We treat each type of $\gamma$
separately. If $\gamma$ is elliptic, then
$$
   |\gamma^k|-\sum_{i=1}^s|\gamma^{k_i}|
   = 2r(k-\sum_{i=1}^sk_i) + 2([k\alpha]-\sum_{i=1}^s[k_i\alpha]) \geq
   0
$$
because $[a+b]\geq[a]+[b]$ for any real numbers $a,b$.
If $\gamma$ is even hyperbolic, then
$$
   |\gamma^k|-\sum_{i=1}^s|\gamma^{k_i}|
   = (2kr-1) - \sum_{i=1}^s(2k_ir-1) = s-1 \geq 0.
$$
Finally, if $\gamma$ is odd hyperbolic, then
$$
   |\gamma^k|-\sum_{i=1}^s|\gamma^{k_i}|
   = \bigl((2r+1)k-1\bigr) - \sum_{i=1}^s\bigl((2r+1)k_i-1\bigr) = s-1
   \geq 0.
$$
\end{proof}

\subsection{Torus dynamics}
Consider the $n$-dimensional torus $\T^n=\R^n/\Z^n$. Let
$(\xi_1,\dots,\xi_n)\in[0,1]^n$. Let $\tau\colon\T^n\to\T^n$ denote
translation by $(\xi_1,\dots,\xi_n)$.

\begin{lemma}\label{l:torus}
The orbit $\{\tau^m(0)\}_{m\in\Z}$ of $0\in\T^n$ under $\tau$ is dense
in a finite collection of translates of an $l$-torus $\T^l$ which is a
subgroup of $\T^n$, where
$$
l+1 = \dim_{\Q}\Bigl(\Spa_\Q(\xi_1,\dots,\xi_n,1)\Bigr).
$$
Here $\Spa_\Q(\xi_1,\dots,\xi_n,1)$ denotes the vector subspace of
$\R$ spanned by $\xi_1,\dots,\xi_n,1$, where $\R$ is viewed as a
vector space over $\Q$.
\end{lemma}

\begin{remark}
This lemma is a discrete version of the well-known fact
that the geodesic through $0\in T^m$ in direction
$(\xi_1,\dots,\xi_m)$ is dense in a subtorus $T^k\subset T^m$, where
$$
   k = \dim_{\Q}\Bigl(\Spa_\Q(\xi_1,\dots,\xi_m)\Bigr).
$$
\end{remark}

\begin{proof}
Assume first that $l=n$. Then by the preceding remark the geodesic
$\gamma$ through $0\in T^{n+1}$ in direction $(\xi_1,\dots,\xi_n,1)$
is dense in
the torus $T^{n+1}$. Since $\gamma$ is transverse to the subtorus
$T^n$ where the last coordinate equals zero, the intersection
$\gamma\cap T^n$ is dense in $T^n$. But this intersection is just the
orbit $\{\tau^m(0)\}_{m\in\Z}$ of $0$ in $\T^n$.


Consider next the case when $l<n$. In this case there is an equation
of the form
$$
m_1\xi_1+m_2\xi_2+\dots+m_n\xi_n+ m_{n+1}=0,
$$
where all $m_j$ are integers. If $d$ is the greatest common divisor of
$m_1,\dots,m_n$ we may rewrite this as
$$
m_1'(d\xi_1)+\dots+m_n'(d\xi_n)+m_{n+1}=0.
$$
Since $m_1',\dots,m_n'$ do not have any common divisor, it follows
that the hyperplane $H$ in $\R^n$ given by the equation
$$
m_1'(dx_1)+\dots+m_n'(dx_n)+m_{n+1}=0
$$
contains a point with integer coordinates. It follows that any iterate
$\tau^{kd}(0)$ which is a multiple of $d$ lies in the torus
$\T^{n-1}\subset \T^n$ which is the subgroup with preimage in $\R^n$
given by the integer translates of the hyperplane $H$. Any other
iterate lies in a translate of this subgroup by $\tau^j(0)$,
$j=1,\dots, d-1$.

To finish the proof we use induction. The intersection of the lattice
$\Z^n$ and the hyperplane $H$ is again a lattice generated by vectors
with integer coordinates. Let $v_1,\dots,v_{n-1}$ be a basis. Writing
$$
(d\xi_1,\dots,d\xi_n)=\eta_1 v_1 +\dots+\eta_{n-1} v_{n-1},
$$
we find that the vector spaces over $\Q$ spanned by
$\eta_1,\dots,\eta_{n-1}$ and the vector space spanned by
$\xi_1,\dots,\xi_n$ are equal. Hence
$$
\dim_{\Q}\Bigl(\Spa_{\Q}(\eta_1,\dots,\eta_{n-1},1)\Bigr)=l.
$$
Using the argument above we can either confine the orbits to
translates of lower dimensional tori (if $l<n-1$) or the orbit is
dense (if $l=n-1$). The lemma follows after $n-l$ steps.
\end{proof}

\subsection{Jump sequences}\label{sec:jump}

To each real number $\xi\in(0,1)$ we associate its {\em jump sequence}
$j(\xi)=\{j_n(\xi)\}_{n\in\N}$ via
$$
   j_n(\xi) = [n/\xi],\qquad n\in\N.
$$
Thus $j_n=j_n(\xi)$ is the unique integer satisfying
$$
   j_n\xi\leq n < (j_n+1)\xi,
$$
so the $n$-th jump in the sequence $\{[k\xi]\}_{k\in\N}$ occurs at
$k=j_n(\xi)$. The jump sequence $j(\xi)$ determines $\xi$ via
$$
   \lim_{n\to\infty}\frac{n}{j_n(\xi)}=\xi.
$$

\begin{lemma}\label{lem:jump}
For $i=1,2,3$ let $\xi_i\in(0,1)$ be irrational with jump sequences
$j(\xi_i)$.
\begin{itemize}
\item[{\rm (a)}] If $j(\xi_2)$ is a subsequence of $j(\xi_1)$, then there exist a
linear relation
$$
   \xi_2=p\xi_1+q,\qquad p,q\in\Q,\quad p>0.
$$
\item[{\rm (b)}] If $j(\xi_2)$ and $j(\xi_3)$ are both subsequences of $j(\xi_1)$,
then $\xi_2$ and $\xi_3$ have a common jump, i.e.~there exist
$m_2,m_3\in\N$ such that $j_{m_2}(\xi_2)=j_{m_3}(\xi_3)$.
\end{itemize}
\end{lemma}

\begin{example}
The situation in (a) occurs e.g.~for $\xi_2=\xi_1/k$, $k\in\N$, in
which case $j_n(\xi_2)=j_{nk}(\xi_1)$. An example with $q\neq 0$ is
given by $\xi_2=\xi_1-\frac12$: If $k=j_n(\xi_2)$ is in the jump
sequence of $\xi_2$ then
$$
k\xi_1< n+\frac{k}{2}< (k+1)\xi_1 -\frac12.
$$
It follows that for $k$ even,
$$
k\xi_1< n+\frac{k}{2}< (k+1)\xi_1,
$$
and for $k$ odd,
$$
k\xi_1< n+\frac{k}{2}+\frac12< (k+1)\xi_1.
$$
Hence in either case $k$ is in the jump sequence of $\xi_1$.
\end{example}

\begin{remark}\label{rem:jump}
D.~Kotschick has pointed out that if $j(\xi_2)$ is a subsequence of
$j(\xi_1)$ and $\xi_1\leq 1/2$, then $\xi_1=k\xi_2$ for some
$k\in\N$. This can be seen as follows. Suppose
$j_n(\xi_2)=j_{m(n)}(\xi_1)$ for a sequence $m:\N\to\N$. It easily
follows that $m$ is a quasi-morphism with defect (deviation from being
a semi-group homomorphism) at most $[3\xi_1]\leq 1$. A more careful
estimate shows that for $\xi_1\leq 1/2$ the defect equals zero, so $m$
is a group homomorphism, i.e.~$m(n)=kn$ for some $k\in\N$.

As the preceding example shows, this result fails as soon as
$\xi_1>1/2$. It might be interesting to characterize all pairs
$(\xi_1,\xi_2)$ for which $j(\xi_2)$ is a subsequence of $j(\xi_1)$.
\end{remark}

\begin{proof}[Proof of Lemma~\ref{lem:jump}]
Consider (a). Let $\T^2=\R^2/\Z^2$ and let
$\tau\colon\T^2\to\T^2$ be translation by the vector
$(\xi_1,\xi_2)$. Lemma \ref{l:torus} implies
that if $\xi_1$, $\xi_2$ and $1$ are linearly independent
over $\Q$, then the orbit of $0$ under $\tau$ is dense in
$\T^2$. Thus there exists an iterate
$\tau^k(0)=(x_1,x_2)\in(0,1)^2$ such that $x_1+\xi_1<1$ and
$x_2+\xi_2>1$. So $k$ is a jump of $\xi_2$ but not of $\xi_1$,
contradicting the hypothesis. We conclude that $\xi_1$, $\xi_2$ and
$1$ satisfy a linear relation
$$
   \xi_2=p\xi_1+q,\qquad p,q\in\Q.
$$
The proof of Lemma~\ref{l:torus} shows that the orbit of $0$ is dense
in the straight line
through $0$ with slope $p$. If $p<0$, this implies that there is an
iterate $\tau^k(0)$ with representative
$(x_1,x_2)\in(-\frac12,\frac12)^2$ arbitrarily close to $0$ and such
that $x_1>0$, $x_2<0$. So $k$ is a jump of $\xi_2$ but not of $\xi_1$,
contradicting the hypothesis. Hence the slope $p$ is positive.

Consider (b). Part (a) implies that there are linear relations
$$
   \xi_j=p_j\xi_1+q_j,\qquad p_j,q_j\in\Q,\quad p_j>0,\quad j=2,3.
$$
It follows that there is a linear relation
$$
   \xi_3=p\xi_2+q,\qquad p,q\in\Q,\quad p>0.
$$
Again by the proof of Lemma~\ref{l:torus}, the orbit of $0$ under the
translation
$\tau:\T^2\to \T^2$ by $(\xi_2,\xi_3)$ is dense in the straight line
through $0$ with slope $p$. This implies that there is an iterate
$\tau^k(0)$ with representative $(x_1,x_2)\in(-\frac12,\frac12)^2$
arbitrarily close to $0$ and such that $x_1<0$, $x_2<0$. So $k$ is a
jump of both $\xi_2$ and $\xi_3$.
\end{proof}

\subsection{Proof of Theorem~\ref{thm:two}}
Assume that the differential in linearized contact homology
vanishes. Then in view of the discussion preceding
Theorem~\ref{thm:two}, there is a unique (not
necessarily simple) closed Reeb orbit of every positive even degree
and no other Reeb orbits. In particular, there are no
even hyperbolic orbits. We distinguish three cases.

{\bf Case 1: }There are at least two simple Reeb orbits which are odd
hyperbolic.

In this case, the formula for the Conley-Zehnder index shows that
suitable odd multiplies of the simple orbits have the same
degree. This contradicts vanishing differential. Hence, Case 1 is
ruled out.

{\bf Case 2: }There is exactly one simple Reeb orbit which is odd
hyperbolic.

Denote this simple orbit by $\gamma_1$. Write $\CZ(\gamma_1)=2r_1+1$.
Dynamical convexity implies $r_1\ge 1$. Odd multiples
$\gamma_1^{2\ell-1}$ of $\gamma_1$ have degrees
$|\gamma_1^{2\ell-1}|=(2r_1+1)(2\ell-1)-1$. Since
$\{(2r_1+1)(2\ell-1)-1\}_{\ell=1}^\infty$  does not contain all
positive even integers, there exists some other orbit of even
degree. Since $\gamma_1$ is assumed to be the only odd hyperbolic
orbit we conclude that there exists an elliptic orbit. Let $\gamma_2$
be a simple elliptic orbit. Write
$\CZ(\gamma_2)=2[r_2+\alpha_2]+1=2r_2+1$, where $r_2\in\Z$ and
$\alpha_2\in(0,1)$ is irrational. Dynamical convexity
implies $r_2\ge 1$. Multiples $\gamma_2^k$ of $\gamma_2$ have degrees
$2[k(r_2+\alpha_2)]$. We claim that there exists multiples $k$ and
$\ell$ such that
\begin{equation}\label{eq:irr1}
   (2r_1+1)(2\ell-1)-1 = 2[k(r_2+\alpha_2)].
\end{equation}
This claim implies that $|\gamma_1^{2\ell-1}|=|\gamma_2^k|$ which
contradicts vanishing differential and allows us to rule out Case 2.

To verify the claim we argue as follows. Let $v=2r_1+1$ and note that
the numbers in the left hand side of \eqref{eq:irr1} can be written as
$(\ell-1)(2v)+(v-1)$, $\ell=1,2,\dots$. Consider the circle
$\R/(v\cdot\Z)$ and the irrational rotation
$\tau\colon\R/(v\cdot\Z)\to\R/(v\cdot\Z)$,
$\tau(x)=x+(r_2+\alpha_2)$. Lemma \ref{l:torus} implies that there
exists $k$ such that
$$
\dist\left(\tau^k(0),\left(\frac{v-1}{2}+\frac14\right)\right)<10^{-80}.
$$
This means that there exists $\ell\ge 1$ such that
$$
\left|(\ell-1)v +
  \frac{v-1}{2}+\frac14-k(r_2+\alpha_2)\right|<10^{-80}.
$$
It follows that
$\left|((2r_1+1)(2\ell-1)-1+\frac12)-2k(r_2+\alpha_2)\right|<10^{-79}$.
This implies the claim.

{\bf Case 3: }All simple Reeb orbits are elliptic.

Let $\gamma_j$, $j=1,2,\dots$ denote the simple orbits. Write
$\CZ(\gamma_j)=2[r_j+\alpha_j]+1=2r_j+1$ where $r_j\in\Z$ and
$\alpha_j\in(0,1)$ is irrational and note that dynamical convexity
implies $r_j\ge 1$. Multiples $\gamma_j^k$ of $\gamma_j$ have degrees
$2[k(r_j+\alpha_j)]$. For fixed $j$ the set
$\bigl\{2[k(r_j+\alpha_j)]\bigr\}_{k=1}^\infty$ does not contain all
even multiples of $r_j$: For every multiplicity $k$ such that there
exists an integer $m$ with
$$
k\alpha_j< m< (k+1)\alpha_j,
$$
the equation
$$
|\gamma_j^{k+1}|-|\gamma_j^k| =
2[(k+1)(r_j+\alpha_j)]-2[k(r_j+\alpha_j)]=2(r_j+1)\geq 4
$$
holds. We conclude that there are at least two simple orbits.

Now we invoke a result proved in~\cite{BO}:
For all $k\in\N$ there exist isomorphisms $HC^\lin_{2k+2}(S^3)\cong
HC^\lin_{2k}(S^3)$ between linearized contact homology induced by
chain maps counting holomorphic curves in the symplectization. Since
the $d\lambda$-area of
a holomorphic curve is positive, this implies the following, in the
situation under consideration:
\begin{itemize}
\item[(O)]
The two orderings of the set $\{\gamma_j^k\}$, $j,k=1,2,\dots$ of all
closed Reeb orbits by increasing degree and by increasing action,
respectively, coincide.
\end{itemize}

If $\AA_j$ denotes the action of the simple Reeb orbit $\gamma_j$, then
the action of $\gamma_j^k$ equals $k\AA_j$. It follows from (O) that
each action ratio $\frac{\AA_i}{\AA_j}$, $i\ne j$, must be irrational:
if it were rational some multiples of $\gamma_i$ and $\gamma_j$ would
have the same action and hence the same degree, which is impossible
since the differential vanishes.

Let $\gamma_1$ denote the simple orbit of smallest action. It follows
from (O) and the fact that the minimal degree is $2$ that
$|\gamma_1|=2[r_1+\alpha_1]=2r_1=2$. Hence $r_1=1$ and the degrees
of adjacent multiples of $\gamma_1$ differ either by $2$ or by
$4$. Consider a simple orbit $\gamma_j$, $j\ne 1$. Since
$\frac{\AA_1}{\AA_j}\in(0,1)$ is irrational, for any integer $m$ there
is a unique integer $k$ such that $k\AA_1<m\AA_j<(k+1)\AA_1$, or
equivalently
\begin{equation}\label{e:betwj}
k\frac{\AA_1}{\AA_j}< m < (k+1)\frac{\AA_1}{\AA_j}.
\end{equation}
In particular, (O) implies that $|\gamma_1^{k+1}|-|\gamma_1^k|=4$, or
equivalently
\begin{equation}\label{e:betw1}
k\alpha_1< m'<(k+1)\alpha_1,
\end{equation}
for some integer $m'$. In the language of Section~\ref{sec:jump}, this
means:
\begin{itemize}
\item[(J)]
For each $j>1$ the jump sequence of $\xi_j=\frac{\AA_1}{\AA_j}$ is a
subsequence of the jump sequence of $\xi_1=\alpha_1$.
\end{itemize}

Now we prove that there are at most two simple orbits. To see this,
assume that there are three distinct orbits $\gamma_j$,
$j=1,2,3$. Condition (J) and Lemma~\ref{lem:jump} (b) imply that
$\xi_j=\frac{\AA_1}{\AA_j}$, $j=2,3$, have a common jump, i.e.~that there
exist integers $k$, $m_2$, and $m_3$ such that
\begin{align*}
&k\frac{\AA_1}{\AA_2}< m_2<(k+1)\frac{\AA_1}{\AA_2},
&k\frac{\AA_1}{\AA_3}< m_3<(k+1)\frac{\AA_1}{\AA_3}.
\end{align*}
This implies that both $m_2\AA_2$ and $m_3\AA_3$ lie between
$k\AA_1$ and $(k+1)\AA_1$. Since $|\gamma_1^{k+1}|-|\gamma_1^k|\leq
4$, it follows from (O) that $|\gamma_2^{m_2}|=|\gamma_3^{m_3}|$,
which contradicts vanishing differential. Hence there are at most two
simple Reeb orbits. Since we already know that there are at least
two orbits, this proves that (i) implies (ii).

For the converse implication suppose that the differential in
linearized contact homology does not vanish. This can only happen if
there is at least one even hyperbolic orbit $\gamma_1$ (whose iterates
have odd degree) and one elliptic or odd hyperbolic orbit $\gamma_2$
(whose good iterates have even degree). Dynamical convexity implies $r_2>0$ and the iteration
formulae for Conley-Zehnder indices show that multiples of $\gamma_2$
cannot attain all even degrees, so there must be a third orbit. Thus (ii) implies (i).

Next suppose that (i) and (ii) hold, so there are precisely two simple
closed orbits $\gamma_1,\gamma_2$ with actions $\AA_i$, Floquet
multipliers $\alpha_i$ and rotation numbers $r_i$. We first show that
these data are realized by a suitable irrational ellipsoid.

Since there are only two simple closed orbits, condition (O)
implies
that the jump sequences of
$\xi_1=\alpha_1$ and $\xi_2=\frac{\AA_1}{\AA_2}$ agree. Since the
$\xi_i$ are determined by their jump sequences, it follows that
\begin{equation}\label{eq:a1}
   \alpha_1=\frac{\AA_1}{\AA_2}.
\end{equation}
Next note that (O) implies that for each $m$, if $k$ is such that
$$
k\AA_1< m\AA_2<(k+1)\AA_1,
$$
then, since $r_1 = 1$,
$$
[k(1+\alpha_1)]< [m(r_2+\alpha_2)]<[(k+1)(1+\alpha_1)].
$$
These inequalities combine to
$$
\frac{[k(1+\alpha_1)]}{k+1} < \frac{[m(r_2+\alpha_2)]\AA_1}{m\AA_2}
< \frac{[(k+1)(1+\alpha_1)]}{k},
$$
which in the limit $m\to\infty$ (hence $k\to\infty$) yields
\begin{equation}\label{eq:a2}
\frac{\AA_2}{\AA_1}=\frac{r_2+\alpha_2}{1+\alpha_1}.
\end{equation}
Next note that the first gap in the degrees of multiples of
$\gamma_1$ occurs at the degree $2([1/\alpha_1]+1)$. So this must be
the degree of $\gamma_2$ and we conclude
\begin{equation}\label{eq:r}
   r_2=[1/\alpha_1]+1.
\end{equation}
Equations~\eqref{eq:a1}, \eqref{eq:a2} and~\eqref{eq:r} together with
$r_1=1$ uniquely determine $\alpha_1$, $\alpha_2$, $r_1$ and $r_2$ in
terms of the action ratio $\AA_2/\AA_1$:
\begin{equation}\label{eq:ra}
   \alpha_1=\frac{\AA_1}{\AA_2},\qquad
   \alpha_2 = \frac{\AA_2}{\AA_1} -
   \left[\frac{\AA_2}{\AA_1}\right],\qquad
   r_1=1,\qquad r_2=\left[\frac{\AA_2}{\AA_1}\right]+1.
\end{equation}

It remains to show that any combination of actions $\AA_i$, Floquet
multipliers $\alpha_i$ and rotation numbers $r_i$ satisfying
equations~\eqref{eq:ra} is realized by an ellipsoid
$$
   E = \left\{(z_1,z_2)\in\C^2\left|\, \frac{|z_1|^2}{a_1} +
   \frac{|z_1|^2}{a_1} = 1\right.\right\}
$$
with suitable $0<a_1<a_2$ such that $a_1/a_2$ is irrational.

The Reeb flow on $E$ is given by $z_i(t)=e^{2it/a_i}$, which is
periodic of period $T_i=\pi a_i$ in the $i$-th component. Thus for
$a_1/a_2$ irrational there are precisely two simple closed orbits
$\gamma_1=\{z_2=0\}$ and $\gamma_2=\{z_1=0\}$ of period (= action)
$$
   \AA_i = T_i = \pi a_i.
$$
Define the $a_i$ by this equation. From
$$
   z_2(T_1)=e^{2\pi ia_1/a_2},\qquad z_1(T_2)=e^{2\pi ia_2/a_1}
$$
we read off the Floquet multipliers
$$
   \alpha_1=a_1/a_2,\qquad \alpha_2=a_2/a_1-[a_2/a_1]
$$
and the rotation numbers
$$
   r_1=1,\qquad r_2=[a_2/a_1]+1
$$
(the additional $+1$ result from the choice of trivializations that
extend over disks). For $a_i=\AA_i/\pi$ these equations agree with
equations~\eqref{eq:ra}.

It is proved in~\cite{HWZ} that for dynamically convex contact forms there exists a
simple closed orbit $\gamma_1$ with $\CZ(\gamma_1)=3$ (unique in our
case) which is the binding of an open book decomposition. The pages
are discs whose interiors are transverse to the Reeb vector field.
The return map of a page possesses at least one fixed point (see~\cite{HWZ}),
which in our case must correspond to the second simple closed orbit
$\gamma_2$. Hence $\gamma_1$ and $\gamma_2$ are both
unknotted and have linking number $1$, which concludes the proof of
Theorem~\ref{thm:two}.\qed

\begin{remark}\label{r:dynconv}
In this remark we assume abstract perturbations exist so that
the linearized contact homology is well defined for all tight
contact forms on $S^{3}$ with nondegenerate closed Reeb orbits. As
mentioned in Remark
\ref{r:sp=gen}, it is immediate from the formulas for Conley-Zehnder
indices together with the computation of the linearized contact
homology of a tight contact form on $S^{3}$ that condition (i) in
Theorem \ref{thm:two} implies dynamical convexity. We show that (ii)
implies (i) and hence (ii) implies dynamical convexity as well.

Suppose that the differential in linearized contact homology does not
vanish. This can only happen if
there is at least one even hyperbolic orbit $\gamma_1$ (whose iterates
have odd degree) and one elliptic or odd hyperbolic orbit $\gamma_2$
(whose good iterates have even degree). If $r_2\ne 0$ then we conclude
that (i) holds exactly as in the proof above: the iteration formulas
for Conley-Zehnder indices show that multiples of $\gamma_2$
cannot attain all even degrees, so there must be a third orbit. If
$r_2 = 0$ then $|\gamma_2| = 0$, so we must have $|\gamma_1| = \pm
1$.
If $|\gamma_1|=-1$, then $|\gamma_1^k|=-1$ for all $k \ge 1$, so that
infinitely many orbits are required to eliminate these generators in
homology. If $|\gamma_1| = 1$, then $|\gamma_1^k| = 2k-1$ for all
$k \ge 1$, so that in order to eliminate these generators in
homology and span exactly $HC^\lin(S^3)$ in degrees $2$ to $2k-2$,
we must have $|\gamma_2^{2k-2}| < 2k-1 < |\gamma_2^{2k}|$ for all $k > 1$.
This implies that $\gamma_2$ is elliptic and that
$4 \alpha_2 = \lim_{k\to\infty} \frac{|\gamma_2^{2k}|}k = 2$,
a contradiction. This proves that (ii) implies (i).
\end{remark}


\end{document}